\documentclass{amsart}
\usepackage{graphicx} 
\usepackage{proof}

\title{On an inferential semantics for \\ intuitionistic sentential logic}

\usepackage{amsthm}
\newtheorem{lemma}{Lemma}

\newtheorem{theorem}{Theorem}
\newtheorem{definition}{Definition}

\usepackage{mathrsfs}
\newcommand{\base}[1]{\mathscr{#1}}
\newcommand{\clause}[1]{\mathcal{#1}}
\usepackage{amssymb}
\newcommand{\supp}{\Vdash}

\newcommand{\setAtom}{\base{B}}

\renewcommand{\emptyset}{\varnothing}
\renewcommand{\phi}{\varphi}

\newcommand{\flatop}[1]{#1^\flat}

\thanks{\emph{Acknowledgements.} I would like to thank Rajeev Gore for introducing me to Mints' work and Tao Gu for his feedback on this paper.}

\author{Alexander V. Gheorghiu}

\address{School of Electronics and Computer Science, University of Southampton\\
University Road, Southampton, SO17 1BJ 
United Kingdom}

\address{Department of Computer Science, University College London\\
Gower St, London WC1E 6BT, UK}

\email{a.v.gheorghiu@soton.ac.uk}
\address{\textbf{ORCID:} 0000-0002-7144-6910}

\begin{document}

\begin{abstract}
Sandqvist’s base-extension semantics (B-eS) for intuitionistic sentential logic grounds meaning relative to bases (rather than, say, models), which are arbitrary sets of permitted inferences over sentences. While his soundness proof is standard, his completeness proof is quite unusual. It closely parallels a method introduced much earlier by Mints, who developed a resolution-based approach to intuitionistic logic using a systematic translation of formulas into sentential counterparts. In this short note, we highlight the connection between these two approaches and show that the soundness and completeness of B-eS follow directly from Mints’ theorem. While the result is modest, it reinforces the relevance of proof-search to proof-theoretic semantics and suggests that resolution methods have a deeper conceptual role in constructive reasoning than is often acknowledged.
\end{abstract}


\maketitle

\section{Introduction}

Sandqvist \cite{Sandqvist2015IL} introduced a \emph{base-extension semantics} (B-eS) for intuitionistic propositional logic, wherein logical signs are defined relative to collections of inference rules for basic sentences, termed \emph{bases}, and consequence is determined via hypothetical \emph{extensions} over such bases. This framework provides an alternative to model-theoretic semantics while preserving soundness and completeness for intuitionistic logic. While the broader motivations for such semantics lie within the research programme of \emph{proof-theoretic semantics} (P-tS), our focus in this paper is on revisiting B-eS and presenting a new proof of the soundness and completeness theorem using techniques developed by Mints for resolution systems.

Sandqvist's proof of soundness follows a standard approach, demonstrating that the semantic judgment relation is preserved across the proof rules of the logic. However, his completeness proof departs from conventional methods. As Sandqvist \cite{Sandqvist2015IL} himself remarks:
\begin{quote}
``The mathematical resources required for the purpose are quite elementary; there will be no need to invoke canonical models, K\"onig’s lemma, or even bar induction. The proof proceeds, instead, by \emph{simulating} an intuitionistic deduction using basic sentences within a base specifically tailored to the inference at hand.''
\end{quote}
Here, the term \emph{basic sentences} refers to atomic formulae distinct from $\bot$, which we denote by $\mathbb{B}$. Indeed, so powerful is this construction that it has been deployed to handle substructural logics with ease~\cite{IMLL,gheorghiu2024proof,BI}. 

The key idea underlying this simulation technique is the systematic translation of formulas into basic sentences. Suppose $\supp \phi$ and that $\phi$ contains a subformula $a \land b$, where $a, b \in \mathbb{B}$. Following the natural deduction rules governing conjunction, Sandqvist makes sure to include the following rules in the ``specifically tailored'' base $\base{N}$ for $\phi$:
\[
\infer{r}{~~a & b~~} \qquad \infer{~~a~~}{r} \qquad \infer{~~b~~}{r}
\]
where $r$ is a fresh basic sentence representing $a \land b$. This means that $r$ behaves in $\base{N}$ as $a \land b$ behaves in Gentzen's NJ~\cite{Gentzen} --- that is, as in intuitionistic sentential logic. 

More generally, each subformula $\chi$ of $\phi$ is assigned a corresponding basic counterpart $\flatop{\chi}$ --- for example, $r = \flatop{(a \land b)}$. The major work is establishing their equivalence within $\base{N}$:
\[
\chi \supp_{\base{N}} \flatop{\chi} \quad \text{and} \quad \flatop{\chi} \supp_{\base{N}} \chi.
\]
Since we assume $\supp \phi$, it follows that $\supp \flatop{\phi}$, and given that every rule in $\base{N}$ corresponds to an intuitionistic natural deduction rule, we conclude $\vdash \phi$, as required.

The \emph{flattening} technique employed here is strikingly different from standard completeness proofs. However, as the saying goes, ``\emph{all theorems were already proved in the Soviet Union}''. In particular, this approach bears a strong resemblance to a method developed by Mints~\cite{Mints1981,Mints1985,Mints1986,Mints1987,Mints1990} from the 1980s for evaluating intuitionistic consequence via resolution systems. Mints' approach systematically transforms proof systems for various logics, including intuitionistic logic, into resolution-based systems while preserving derivational structure. This extends earlier work by Maslov \cite{Maslov1971} on resolution calculi for classical predicate logic.

Mints defines a class of formulas, \emph{clauses}, with a particularly simple structure (details below). Given a formula $\phi$, he constructs a set of clauses $M$ such that for some designated basic sentence $g$,
\[
\vdash \phi \quad \text{iff} \quad M \vdash g .\tag{\text{$\dagger$}}
\]
To do this, he associates to each subformula $\chi$ of $\phi$ a basic counterpart $\flatop{\chi}$. The set $M$ is constructed by adding clauses ensuring that $\flatop{\chi}$ correctly encodes the logical behaviour of $\chi$. For instance, if $\chi = \chi_1 \land \chi_2$, the following clauses are introduced:
\[
\flatop{\chi_1} \land \flatop{\chi_2} \to \flatop{\chi}, \quad \flatop{\chi} \to \flatop{\chi_1}, \quad \flatop{\chi} \to \flatop{\chi_2}.
\]
The key observation is that relative to such clauses,
\[
\flatop{(\chi_1 \land \chi_2)} \leftrightarrow \flatop{\chi_1} \land \flatop{\chi_2}.
\]
Thus, reasoning about $\phi$ reduces to reasoning about formulas with at most three atomic components. In ($\dagger$), $M$ is the clause set encoding $\phi$ and $g = \flatop{\phi}$.

There are clear similarities even at the level of detail presented here. There is a natural correspondence between bases and clausal systems such that for any formula $\phi$, Sandqvist's $\base{N}$ and Mints' $M$ coincide when identical atomic names are used for subformulas. More significantly, this bijection establishes a pointwise correspondence between support in a base and proof-search in a clause set, aligning validity in B-eS for $\phi$ with resolution in $M$ for $g$. Consequently, the soundness and completeness of intuitionistic logic with respect to Sandqvist's B-eS \textit{both} follow as corollaries of Mints' theorem.

\section{Sandqvist's Base-extension Semantics}

We present a concise but complete formulation of the B-eS for intuitionistic logic as introduced by Sandqvist~\cite{Sandqvist2015IL}. The fundamental notion in this framework is that of \emph{derivability in a base}, which serves as the foundation for making assertions. We defer detailed discussion of the motivations behind B-eS to elsewhere --- see, for example, Sandqvist~\cite{Sandqvist2005inferentialist,Sandqvist2009CL,Sandqvist2015IL} and de Campos Sanz al.~\cite{Piecha2015failure}, and Gheorghiu~\cite{ggp2024practice,Gheorghiu2024ProofTheoretic}.

A \emph{base} consists of a collection of inference rules over basic sentences --- that is, atomic formulae other than $\bot$, whose collection is denoted $\mathbb{B}$. Rather than employing standard natural deduction notation, we introduce an inline notation for clarity.

\begin{definition}[Atomic Rule]
    An \emph{atomic rule} is either:
    \begin{itemize}
        \item a \emph{nullary rule} of the form $\epsilon \Rightarrow c$, or
        \item a \emph{general rule} of the form 
        \[
        (P_1 \Rightarrow p_1),\dots,(P_n \Rightarrow p_n) \Rightarrow c,
        \]
        where $P_1, \dots, P_n$ are finite sets of basic sentences, and $p_1, \dots, p_n, c$ are basic sentences.
    \end{itemize}
    If $P_i = \emptyset$ for all $i = 1, \dots, n$, we abbreviate this as $p_1, \dots, p_n \Rightarrow c$.
\end{definition}

\begin{definition}[Base]
    A \emph{base} $\base{B}$ is a set of atomic rules.
\end{definition}

Intuitively, an inference
\[
\infer{c}{\deduce{p_1}{[P_1]} & \dots & \deduce{p_n}{[P_n]}}
\]
corresponds to the atomic rule $(P_1 \Rightarrow p_1),\dots,(P_n \Rightarrow p_n) \Rightarrow c$. Unlike standard natural deduction systems, these rules are considered \emph{as given}, meaning that no substitution of propositional letters is permitted when applying them.

\begin{definition}[Derivability in a base] \label{def:derbase}
   Given a base $\base{B}$, the derivability relation $\vdash_{\base{B}}$ is the smallest relation satisfying the following conditions for any $P \subseteq \mathbb{B}$:
   \begin{itemize}
       \item If $p \in P$, then $P \vdash_{\base{B}} p$.
       \item If $(\epsilon \Rightarrow c) \in \base{B}$, then $P \vdash_{\base{B}} c$.
       \item If $\big((P_1 \Rightarrow p_1),\dots,(P_n \Rightarrow p_n) \Rightarrow c\big) \in \base{B}$ and $P, P_i \vdash_{\base{B}} p_i$ for all $i = 1, \dots, n$, then $P \vdash_{\base{B}} c$.
   \end{itemize}
\end{definition}
The semantic judgment in B-eS is defined in terms of a \emph{support} relation, which plays a role analogous to satisfaction in model-theoretic semantics.

\begin{definition}[Support] \label{def:supp}
   A formula $\phi$ is \emph{supported} in a base $\base{B}$ relative to a context $\Gamma$ if $\Gamma \supp_{\base{B}} \phi$, as defined in Figure~\ref{fig:sandqvist}. It is \textit{supported} --- denoted              $\Gamma \supp \phi$ ---  iff $\Gamma \supp_{\base{B}} \phi$ for any base $\base{B}$.
\end{definition}

\begin{figure}[t]
\hrule 
\vspace{2mm}
 \[
        \begin{array}{l@{\quad}c@{\quad}l@{\quad}r}
            \supp_{\base{B}} p  & \text{iff} &   \vdash_{\base{B}} p & \text{(At)}  \\[1mm]
             \supp_{\base{B}} \phi \to \psi & \text{iff} & \phi \supp_{\base{B}} \psi & (\to) \\[1mm]
             \supp_{\base{B}} \phi \land \psi   & \text{iff} & \supp_{\base{B}} \phi \text{ and } \supp_{\base{B}} \psi & (\land) \\[1mm]
             \supp_{\base{B}} \phi \lor \psi & \text{iff} &  \forall \base{C} \supseteq \base{B}, \forall p \in \setAtom, \\
             & & \text{if } \phi \supp_{\base{C}} p \text{ and }  \psi \supp_{\base{C}} p, \text{ then } \supp_{\base{C}} p & (\lor)  \\[1mm]
             \supp_{\base{B}} \bot & \text{iff} &    \supp_{\base{B}} p \text{ for any } p \in \setAtom & (\bot) \\[1mm]
           \Gamma \supp_{\base{B}} \phi & \text{iff} & \forall \base{C} \supseteq \base{B}, \text{ if } \supp_{\base{C}} \psi \text{ for any } \psi \in \Gamma, \text{ then } \supp_{\base{C}} \phi &  (\text{Inf})
        \end{array}
 \]
 \vspace{2mm}
\hrule
\caption{Sandqvist's Base-Extension Semantics} 
\label{fig:sandqvist}
\end{figure}

This is an inductive definition, but the induction measure is not the size of the formula $\phi$. Instead, it is a measure of the \textit{logical weight} $w$ of $\phi$. Following Sandqvist~\cite{Sandqvist2015IL},
\[
w(\phi) := 
\begin{cases}
    0 & \mbox{if $\phi \in \setAtom$} \\
    1 & \mbox{if $\phi = \bot$} \\
   w(\phi_1)+w(\phi_2)+1 & \mbox{if $\phi = \phi_1 \circ \phi_2$ for any $\circ \in \{\to, \land, \lor\}$}
\end{cases}
\]
In each clause of Figure~\ref{fig:sandqvist}, the sum of the weights of the formulas
flanking the support judgment in the definiendum exceeds the corresponding number for any occurrence of the judgment in the definiens. We refer to induction relative to this measure as \emph{semantic} induction to distinguish it from \emph{structural} induction.

A key feature of B-eS is its resemblance to Kripke semantics~\cite{kripke1965semantical}, particularly in employing an ``extension'' relation to handle consequence. However, its treatment of disjunction ($\lor$) differs significantly, recalling the second-order characterization of Prawitz~\cite{Prawitz2006natural}. Piecha et al.~\cite{Piecha2015failure} have shown that using a conventional treatment of disjunction, 
\[
\supp_{\base{B}} \phi \lor \psi \qquad \text{iff} \qquad \supp_{\base{B}} \phi \text{ or } \supp_{\base{B}} \psi,
\]
results in the incompleteness of intuitionistic sentential logic. Indeed, Stafford~\cite{Stafford2021} has demonstrated that the resulting semantics corresponds to an intermediate logic known as \emph{generalized inquisitive logic}.

Before proceeding, we state the following elementary lemma:

\begin{lemma}[Sandqvist~\cite{Sandqvist2015IL}]
\label{lem:empty-base}
    $\Gamma \supp \phi$ iff $\Gamma \supp_\emptyset \phi$.
\end{lemma}

This concludes our presentation of the B-eS framework.

\section{Mints' Clausal Representation Theorem}
We present a concise but complete formulation of the result by Mints~\cite{Mints1981,Mints1985,Mints1990} required for this paper. The key concept is that of an \emph{intuitionistic clause}, which has the following syntactic forms:
\[
    (p \to q^*) \to r, \quad  p \to (q \lor r), \quad  (p_1 \land \dots \land p_n) \to q^*,
\]
where:
\begin{itemize}
    \item $p, q, r, p_1, \dots, p_n$ are basic propositions;
    \item $q^*$ denotes either $q$ or $\bot$ (absurdity).
\end{itemize}
A clause is \emph{basic} if it contains only basic propositions. A (basic) \emph{clausal system} is a collection of (basic) clauses.

Let $\phi$ be a formula, and assume without loss of generality that it does not contain occurrences of $\bot$ except as the conclusion of an implication (i.e., only in subformulas of the form $\chi \to \bot$). Define $\Xi$ as the set of subformulas of $\phi$. We introduce an injection $\flatop{(-)}: \Xi \to \mathbb{B}$ such that $\flatop{p} = p$ whenever $p \in \Xi \cap \mathbb{B}$.

Let $\mathbb{X}$ be a set of basic sentences. The clausal translations for complex subformulas of $\phi$ are defined as follows:
\begin{align*}
    \clause{M}_\mathbb{X}(\phi_1 \land \phi_2) &:= \{ (\flatop{(\phi_1 \land \phi_2)} \to \flatop{\phi_1}), (\flatop{(\phi_1 \land \phi_2)} \to \flatop{\phi_2}), \\ 
    &\quad (\flatop{\phi_1} \land \flatop{\phi_2} \to \flatop{(\phi_1 \land \phi_2)}) \}, \\
    \clause{M}_\mathbb{X}(\phi_1 \lor \phi_2) &:= \{ (\flatop{\phi_1} \to \flatop{(\phi_1 \lor \phi_2)}), (\flatop{\phi_2} \to \flatop{(\phi_1 \lor \phi_2)}) \} \\
    &\quad \cup \{(\flatop{(\phi_1 \lor \phi_2)} \land (\flatop{\phi_1} \to x) \land (\flatop{\phi_2} \to x) \to x) \mid x \in\mathbb{X} \}, \\
    \clause{M}_\mathbb{X}(\phi_1 \to \phi_2) &:= \{ (\flatop{\phi_1} \land \flatop{\phi_2} \to \flatop{\phi_2}), (\flatop{\phi_1} \to \flatop{\phi_2}) \}.
\end{align*}

The full clausal system is then defined by collecting these clauses along with additional clauses governing $\flatop{\bot}$ (representing $\bot$):
\[
    \clause{M}_\mathbb{X} := \bigcup_{\chi \in \Xi} \clause{M}_\mathbb{X}(\chi) 
    \cup \{(\bigwedge\mathbb{X} \land y \to \flatop{\bot})\} 
    \cup \{ (\flatop{\bot} \to y), (\flatop{\bot} \to x) \mid x \in\mathbb{X} \},
\]
where $\bigwedge\mathbb{X}$ denotes the conjunction of all elements of $X$ and $y$ is some fresh basic sentence. 

\begin{lemma}[Mints' Clausal Theorem~\cite{Mints1981,Mints1985,Mints1990}]
    \label{lem:mints}
    Let $X$ be the range of $\flatop{(-)}$, i.e., $X:= \{\flatop{\chi} \mid \chi \in \Xi\}$. Let $\phi[\flatop{\bot}/\bot]$ be the formula obtained by replacing all occurrences of $\bot$ in $\phi$ with $\flatop{\bot}$. Then:
    \[
        \vdash \phi[\flatop{\bot}/\bot] \quad \text{iff} \quad \clause{M}_\mathbb{X} \vdash \flatop{\phi}.
    \]
\end{lemma}

Observe that if $X$ contains only basic sentences, then the clausal system $\clause{M}_\mathbb{X}$ is basic. We have chosen a specific form of Mints'~\cite{Mints1981,Mints1985,Mints1990} construction suited to our needs, as the resulting clausal systems exhibit a one-to-one correspondence with bases (see below). From a pure proof-search perspective, an alternative clausal representation with a smaller search space could be employed (cf.~\cite{Mints1981,Mints1985}).

This concludes the background on Mints' theorem necessary for this paper. 

\section{Soundness and Completeness}

In this section, we use Mints' theorem to establish the soundness and completeness of intuitionistic sentential logic with respect to Sandqvist's base-extension semantics (B-eS). We begin by proving three simple metatheorems that will be necessary for the proof.

Due to the inference condition (\text{Inf}) in B-eS, we require the following result:

\begin{lemma}
    \label{lem:inf-resolution}
    Let $C$ be a clausal system. Then:
    \[
        p, C \vdash q \quad \text{iff} \quad \text{for any } D \supseteq C, \text{ if } D \vdash p, \text{ then } D \vdash q.
    \]
\end{lemma}

\begin{proof}
    This generalizes Sandqvist's Lemma 2.2~\cite{Sandqvist2015IL}, reformulated in terms of clausal systems rather than bases. The forward direction follows by transitivity. For the reverse direction, suppose $D \supseteq C$ and $D \vdash p$. The result follows by induction on the minimal proof of $D \vdash q$.
\end{proof}

Next, we introduce a generalized version of Mints' theorem, removing restrictions on the consequences of disjunction and absurdity:
\[
\clause{N} := \bigcup_{\chi=\Xi} \clause{M}_\mathbb{B}(\chi) \cup \{(\flatop{\bot} \to x) \mid x \in \mathbb{B}\}.
\]

\begin{lemma}[Modified Mints' Theorem] \label{lem:mod-mints}
For any formula $\phi$:
    \[
    \vdash \phi \quad \text{iff} \quad \clause{N} \vdash \flatop{\phi}.
    \]
\end{lemma}

\begin{proof}
    This follows straightforwardly by induction, showing $\clause{N} \vdash \flatop{\phi}$ iff $\clause{M}_\mathbb{X} \vdash \flatop{\phi}$, where $X$ is the range of $\flatop{(-)}$, and then applying Mints' theorem (Lemma~\ref{lem:mints}).
\end{proof}

Mints' motivation for introducing clausal representations was to develop resolution-based methods for intuitionistic logic. The resolution clauses are designed to define the atoms; for example, relative to $\clause{N}$,
\[
\flatop{(\chi_1 \land \chi_2)} \leftrightarrow (\flatop{\chi_1} \land \flatop{\chi_2})
\]
holds. Formally, this yields the following lemma:

\begin{lemma} \label{lem:reductions}
For any set of clauses $\clause{C}$, the following equivalences hold:
\begin{itemize}
    \item $\clause{C}, \clause{N} \vdash \flatop{(\chi_1 \land \chi_2)}$ iff $\clause{C}, \clause{N} \vdash \flatop{\chi_1} \land \flatop{\chi_2}$.
    \item $\clause{C}, \clause{N} \vdash \flatop{(\chi_1 \lor \chi_2)}$ iff $
    \clause{C}, \clause{N} \vdash (\flatop{\chi_1} \to x) \land (\flatop{\chi_2} \to x) \to x \text{ for } x \in\mathbb{X}$.
    \item $\clause{C}, \clause{N} \vdash \flatop{(\chi_1 \to \chi_2)}$ iff $\flatop{\chi}_1,\clause{C}, \clause{N} \vdash \flatop{\chi_2}$.
    \item $\clause{C}, \clause{N} \vdash \flatop{\bot}$ iff $\clause{C}, \clause{N} \vdash x$ for all $x \in\mathbb{X}$.
\end{itemize}
\end{lemma}

\begin{proof}
    Each case follows by direct inspection of $\clause{N}$. We illustrate the case for disjunction:
    \[
    \clause{C}, \clause{N} \vdash \flatop{(\chi_1 \lor \chi_2)} \quad \text{iff} \quad \clause{C}, \clause{N} \vdash (\flatop{\chi_1} \to x) \land (\flatop{\chi_2} \to x) \to x \text{ for } x \in \mathbb{B}.
    \]
    By construction, the axioms
    \[
    (\flatop{\chi_1} \to \flatop{(\chi_1 \lor \chi_2)}), \quad (\flatop{\chi_2} \to \flatop{(\chi_1 \lor \chi_2)}),
    \]
    and 
    \[
    \flatop{(\chi_1 \lor \chi_2)} \land (\flatop{\chi_1} \to x) \land (\flatop{\chi_2} \to x) \to x
    \]
    are included in $\clause{N} \cup \clause{C}$ for all $x \in \mathbb{B}$. The result follows by \emph{modus ponens}.
\end{proof}

We now define a bijection $[-]$ mapping atomic rules to clauses:
\begin{align*}
[\epsilon \Rightarrow c] &:= c, \\ 
[\big((P_1 \Rightarrow p_1),\dots,(P_n \Rightarrow p_n) \Rightarrow c\big)] &:= (\bigwedge P_1 \to p_1)  \land \dots \land (\bigwedge P_n \to p_n) \to c.
\end{align*}
where $\clause{C},p_1,\dots,p_n$ are basic and $P_i, \dots, P_n$ are finite sets of basic sentences. In the case that $P_i = \emptyset$, we interpret $(\bigwedge P_i \to p_i)$ as $p_i$. This extends to bases $\base{B}$ pointwise:
\[
[\base{B}] := \{ [r] \mid r \in \base{B} \}.
\]
By induction on proof length, we obtain:
\[
[\base{B}] \vdash p \quad \text{iff} \quad \vdash_{\base{B}} p. \tag{$\ast$}
\]
Moreover, $\base{C} \supseteq \base{B}$ iff $[\base{C}] \supseteq [\base{B}]$.

This gives us a smooth way to move between Sandqvist's setup and Mints'. That is, we have the following equivalence:

\begin{lemma} \label{lem:correspondence}
For any base $\base{B}$,
    \[
[\base{B}], \clause{N} \vdash \flatop{\phi} \quad \text{iff} \quad \supp_{\base{B}} \phi
\]
\end{lemma}
\begin{proof}
The result follows by semantic induction on \( \phi \) using \((\ast)\), Lemma~\ref{lem:reductions}, and Lemma~\ref{lem:inf-resolution}. We show the case for disjunction (\( \lor \)), the others being similar.

Suppose \( \phi = \chi_1 \lor \chi_2 \). We proceed as follows. By Lemma~\ref{lem:reductions}, 
\[
[\base{B}], \clause{N} \vdash \flatop{(\chi_1 \lor \chi_2)}
\]
is equivalent to: for $x \in\mathbb{B}$,
\[
[\base{B}], \clause{N} \vdash (\flatop{\chi_1} \to x) \land (\flatop{\chi_2} \to x) \to x .
\]
By basic reasoning in intuitionistic logic, this is equivalent to: for \( x \in\mathbb{B} \),
\[
(\flatop{\chi_1} \to x) \land (\flatop{\chi_2} \to x), [\base{B}], \clause{N} \vdash x .
\]
By Lemma~\ref{lem:inf-resolution} and properties of \( [-] \), this is equivalent to: for \( x \in\mathbb{B} \) and \( \base{C} \supseteq \base{B} \),
\[
[\base{C}], \clause{N} \vdash (\flatop{\chi_1} \to x) \land (\flatop{\chi_2} \to x) \quad \text{implies} \quad [\base{C}], \clause{N} \vdash x .
\]
By further intuitionistic reasoning, this is equivalent to the following: for \( x \in\mathbb{B} \) and \( \base{C} \supseteq \base{B} \),
\[
\flatop{\chi_1}, [\base{C}], \clause{N} \vdash x \mbox{ and } \flatop{\chi_2}, [\base{C}], \clause{N} \vdash  x \quad \text{implies} \quad [\base{C}], \clause{N} \vdash x .
\]
By applications of Lemma~\ref{lem:inf-resolution} and properties of \( [-] \), the \emph{induction hypothesis} (IH) and the definition of support, this is equivalent to: for \( x \in\mathbb{B} \) and \( \base{C} \supseteq \base{B} \), 
\[
(\chi_1 \supp_{\base{C}} x \text{ and } \chi_2 \supp_{\base{C}} x) \quad \text{implies} \quad \supp_{\base{C}} x .
\]
Finally, applying the definition of disjunction (\( \lor \)), this is equivalent to the desired result:
\[
\supp_{\base{B}} \chi_1 \lor \chi_2 .
\]
\end{proof}

These results collectively suffice to prove soundness and completeness of intuitionistic sentential logic with respect to the B-eS: 
\begin{theorem}
    $\supp \phi$ iff $\vdash \phi$.
\end{theorem}
\begin{proof}
We reason as follows:
\begin{align}
    \vdash \phi \qquad &\mbox{iff} \qquad \clause{N} \vdash \flatop{\phi} \tag{Lemma~\ref{lem:mod-mints}} \\
     &\mbox{iff} \qquad \supp_{\emptyset} \phi \tag{Lemma~\ref{lem:correspondence}} \\
     &\mbox{iff} \qquad \supp \phi \tag{Lemma~\ref{lem:empty-base}} 
\end{align}
\end{proof}

\section{Discussion}

We have established a precise correspondence between Sandqvist’s base-extension semantics (B-eS) for intuitionistic sentential logic~\cite{Sandqvist2015IL} and Mints’ resolution systems for intuitionistic logic~\cite{Mints1981,Mints1987,Mints1990}. More specifically, the central idea underpinning Sandqvist’s completeness proof for B-eS is, in essence, Mints’ construction from three decades earlier. Yet, despite this mathematical equivalence, the two frameworks arise from entirely different motivations --- making their connection not only intriguing but also significant.

The B-eS is somewhat enigmatic within proof-theoretic semantics. The field is often motivated by the stance that the rules of a proof calculus, such as Gentzen’s NJ~\cite{Gentzen}, may be thought to `define' the logical constants. This is the philosophical background to proof-theoretic semantics given by, for example,  Gentzen~\cite{Gentzen},  Prawitz~\cite{prawitz1971ideas}, and Dummett~\cite{dummett1991logical}, among others. While there has been some discussion of how the B-eS arises from such rules (cf. ~\cite{ggp2024practice}), no formal account has yet been given.  Thus, while the B-eS is a proof-theoretic semantics in the sense that support reduces to derivability in a base (Definition~\ref{def:derbase}), it remains unclear how the inferential behaviour of the logical connectives is represented.

 Moreover, unlike Kripke's semantics~\cite{kripke1965semantical} for intuitionistic logic, which is typically justified via the metaphor of knowledge accumulation, as articulated by, for example, Dummett~\cite{dummett2000elements}), the B-eS lacks a clear explanation of how it expresses any form of constructivism. In other words: while one may say that the satisfaction relation expresses truth-at-a-state-of-knowledge. There is no apparent account of what the support relation in B-eS is intended to represent.
 
 The correspondence with Mints' work~\cite{Mints1981,Mints1987,Mints1990} begins to clarify these issues. Formally, it demonstrates that support for a formula $\phi$ in a base $\base{B}$ corresponds to proof-search for $\phi$ within a resolution system determined by $\base{B}$. Since these resolution systems preserve the structure of NJ, this reveals that B-eS provides a semantics that is indeed faithful to the inferential behaviour of the logical constants of intuitionistic logic. Moreover, as bases are often read as an agent's set of inferential beliefs, the support relation can thus be understood as the ability of the agent to resolve --- that is, to construct an argument for --- the proposition. So the major result of the technical work herein is that it finally justifies the B-eS within both  the inferentialist traditions of proof-theoretic semantics and within intuitionism.   

Mints’ work~\cite{Mints1981,Mints1987,Mints1990} reveals that resolution is more than a computational or refutational technique --- it is a structural transformation of natural deduction. However, on its own, it does not explain the deeper significance of resolution for constructive reasoning. The connection to the B-eS~\cite{Sandqvist2009CL} for intuitionistic logic in this paper can then be reciprocally interpreted as saying that resolution directly captures the semantic content of the logical constants as determined by their inferential roles since it reconstructs the conditions under which a proposition can be said to be inferentially supported. In this sense, resolution is not merely an algorithmic procedure, but a principled means of articulating what it is to support a proposition constructively.

The broader lesson from this correspondence is that proof-search should play a more explicit role in proof-theoretic semantics. While model-theoretic semantics is only concerned with truth-preservation across deductive inference, the results herein demonstrate that proof-theoretic semantics is not only concerned with proofs but also with the process of finding them --- that is, proof-search. This conceptual shift highlights how proof-theoretic semantics departs from extant semantic paradigms and clarifies how B-eS treats inference as constitutive of meaning, not merely as a mechanism for checking it. This is perhaps not surprising, given the well-documented connections between proof-theoretic semantics and logic programming (cf.~\cite{Gheorghiu2023Definite}). A more systematic study of these links could yield new insights into the mathematical foundations of B-eS and the semantic content of proof-search itself (cf.~\cite{Gheorghiu2025Semantic}).

\bibliographystyle{siam}
\bibliography{bib}

\begin{thebibliography}{10}

\bibitem{dummett1991logical}
{\sc M.~Dummett}, {\em {The Logical Basis of Metaphysics}}, Harvard University Press, 1991.

\bibitem{dummett2000elements}
\leavevmode\vrule height 2pt depth -1.6pt width 23pt, {\em {Elements of Intuitionism}}, Oxford University Press, 2000.

\bibitem{Gentzen}
{\sc G.~Gentzen}, {\em {Investigations into Logical Deduction}}, in {The Collected Papers of Gerhard Gentzen}, M.~E. Szabo, ed., North-Holland Publishing Company, 1969, pp.~68--131.

\bibitem{Gheorghiu2024ProofTheoretic}
{\sc A.~V. Gheorghiu, S.~Ayhan, and V.~Nascimento}, {\em {Proof-theoretic Semantics}}, in Dagstuhl Seminar 24341, Schloss Dagstuhl – Leibniz-Zentrum für Informatik, 2024.

\bibitem{IMLL}
{\sc A.~V. Gheorghiu, T.~Gu, and D.~J. Pym}, {\em {Proof-theoretic Semantics for Intuitionistic Multiplicative Linear Logic}}, in Automated Reasoning with Analytic Tableaux and Related Methods --- TABLEAUX, R.~Ramanayake and J.~Urban, eds., Springer, 2023, pp.~367--385.

\bibitem{BI}
\leavevmode\vrule height 2pt depth -1.6pt width 23pt, {\em {Proof-theoretic Semantics for the Logic of Bunched Implications}}, arXiv:2311.16719,  (2023).
\newblock Accessed February 2024.

\bibitem{ggp2024practice}
\leavevmode\vrule height 2pt depth -1.6pt width 23pt, {\em {A Note on the Practice of Logical Inferentialism}}, in 2nd Logic and Philosophy Conference, 2024.
\newblock arXiv:2403.10546.

\bibitem{gheorghiu2024proof}
{\sc A.~V. Gheorghiu, T.~Gu, and D.~J. Pym}, {\em {Proof-theoretic Semantics for Intuitionistic Multiplicative Linear Logic}}, Studia Logica,  (2024), pp.~1--61.

\bibitem{Gheorghiu2023Definite}
{\sc A.~V. Gheorghiu and D.~J. Pym}, {\em {Definite Formulae, Negation-as-Failure, and the Base-Extension Semantics of Intuitionistic Propositional Logic}}, Bulletin of the Section of Logic, 52 (2023), pp.~239--266.

\bibitem{Gheorghiu2025Semantic}
\leavevmode\vrule height 2pt depth -1.6pt width 23pt, {\em {Semantic Foundations of Reductive Reasoning}}, Topoi,  (2025).
\newblock Special Issue: Meaning and Understanding via Proofs.

\bibitem{kripke1965semantical}
{\sc S.~A. Kripke}, {\em {Semantical Analysis of Intuitionistic Logic I}}, in Studies in Logic and the Foundations of Mathematics, vol.~40, Elsevier, 1965, pp.~92--130.

\bibitem{Maslov1971}
{\sc S.~J. Maslov}, {\em {Connection between the strategies of the inverse method and the resolution method}}, in Seminars in Mathematics, vol.~16, Plenum Publishers, 1971.

\bibitem{Mints1981}
{\sc G.~Mints}, {\em {Resolution calculi for the non-classical logics (in Russian)}}, in 9th Soviet Symposium in Cybernetics, VINITI, 1981, pp.~120--135.

\bibitem{Mints1985}
\leavevmode\vrule height 2pt depth -1.6pt width 23pt, {\em {Resolution calculi for the non-classical logics (in Russian)}}, Semiotics and Informatics, 25 (1985), pp.~120--135.

\bibitem{Mints1986}
\leavevmode\vrule height 2pt depth -1.6pt width 23pt, {\em {Resolution calculi for modal logics (in Russian)}}, Proceedings of the Estonian Academy of Sciences,  (1986), pp.~279--290.

\bibitem{Mints1987}
\leavevmode\vrule height 2pt depth -1.6pt width 23pt, {\em {Cut-free formalisations and resolution methods for propositional modal logic}}, in VIII International Congress for Logic, Methodology and Philosophy of Science, Moscow, 1987, pp.~46--48.

\bibitem{Mints1990}
\leavevmode\vrule height 2pt depth -1.6pt width 23pt, {\em {Gentzen-type systems and resolution rules Part I: Propositional logic}}, in International Conference on Computer Logic --- COLOG, P.~Martin-L{\"o}f and G.~Mints, eds., Springer, 1988, pp.~198--231.

\bibitem{Piecha2015failure}
{\sc T.~Piecha, W.~de~Campos~Sanz, and P.~Schroeder-Heister}, {\em {Failure of Completeness in Proof-theoretic Semantics}}, Journal of Philosophical Logic, 44 (2015), pp.~321--335.

\bibitem{prawitz1971ideas}
{\sc D.~Prawitz}, {\em {Ideas and results in proof theory}}, in Studies in Logic and the Foundations of Mathematics, vol.~63, Elsevier, 1971, pp.~235--307.

\bibitem{Prawitz2006natural}
{\sc D.~Prawitz}, {\em {Natural Deduction: A Proof-theoretical Study}}, Dover Publications, 2006 [1965].

\bibitem{Sandqvist2005inferentialist}
{\sc T.~Sandqvist}, {\em {An Inferentialist Interpretation of Classical Logic}}, PhD thesis, Uppsala University, 2005.

\bibitem{Sandqvist2009CL}
\leavevmode\vrule height 2pt depth -1.6pt width 23pt, {\em {Classical Logic without Bivalence}}, Analysis, 69 (2009), pp.~211--218.

\bibitem{Sandqvist2015IL}
\leavevmode\vrule height 2pt depth -1.6pt width 23pt, {\em {Base-extension Semantics for Intuitionistic Sentential Logic}}, Logic Journal of the IGPL, 23 (2015), pp.~719--731.

\bibitem{Stafford2021}
{\sc W.~Stafford}, {\em {Proof-Theoretic Semantics and Inquisitive Logic}}, Journal of Philosophical Logic, 50 (2021), pp.~1199--1229.

\end{thebibliography}

\end{document}